\documentclass[12pt]{amsart}

\usepackage{amssymb,latexsym}
\usepackage{amsmath,amsfonts,amsthm,amscd,amsxtra,enumerate}
\usepackage{mathtools}
\usepackage[all]{xy}
\usepackage{amsmath,latexsym,amssymb, enumerate, amsthm, epsfig, hyperref,tikz} 
\usepackage[margin=1in]{geometry}
\usepackage{epsfig}

\usepackage[utf8]{inputenc}

\newcommand{\sk}{{\ensuremath{\sf k }}}

\DeclareMathOperator{\Tor}{Tor}

\newtheorem{conjeture}{ Conjecture}[section]
\newtheorem{definition}[conjeture]{ Definition}
\newtheorem{notation}[conjeture]{ Notation}
\newtheorem{example}[conjeture]{ Example}
\newtheorem{theorem}[conjeture]{ Theorem}
\newtheorem{lemma}[conjeture]{ Lemma}
\newtheorem{remark}[conjeture]{ Remark}
\newtheorem{corollary}[conjeture]{ Corollary}
\newtheorem{proposition}[conjeture]{ Proposition}

\renewcommand\dim{\text{\rm dim}}

\providecommand\pdim{\text{\rm pdim}}
\providecommand\Mon{\text{\rm Mon}}

\providecommand\F{{{\sf F}_\bullet}}

\providecommand\Tor{{\rm Tor}}
\providecommand\supp{{\rm supp}}
\providecommand\set{{\rm set}}
\providecommand\X{{\bf X}}
\providecommand\lcm{{\rm lcm}}
\providecommand\hgt{{\rm ht}}
\providecommand\reg{{\rm reg}}
\begin{document}
\title{Regularity, Rees algebra and Betti numbers of certain cover ideals}

\author[A. Kumar]{Ajay Kumar}

\address{Indian Institute of Technology Jammu, India.}

\email{ajay.kumar@iitjammu.ac.in}

\author[R. Kumar]{Rajiv Kumar}

\address{The LNM Institute of Information Technology, Jaipur.}

\email{gargrajiv00@gmail.com}

\date{\today}



\maketitle
\begin{abstract}
	Let $S={\sf k}[X_1,\dots, X_n]$ be a polynomial ring, where ${\sf k}$ is a field. This article deals with the defining ideal of the Rees algebra of squarefree monomial ideal generated in degree $n-2$. As a consequence, we prove that Betti numbers of powers of the cover ideal of the complement graph of a tree do not depend on the choice of tree. Further, we study the regularity and Betti numbers of powers of cover ideals associated to certain graphs. 
\end{abstract}
\section{Introduction}
An interaction between commutative algebra and combinatorics provides new techniques to solve problems in each field. Monomial ideals play an important role to establish a connection between commutative algebra and combinatorics. In particular, given a graph $G$, one can associate a monomial ideal (edge ideal, cover ideal, path ideal), and study algebraic invariants of the corresponding ideal in terms of the combinatorial data of $G$.

Let $G$ be a graph on the vertex set V=$\{X_1,X_2,\dots,X_n \}$, and $S=\sk[X_1,X_2,\dots,X_n]$ be a polynomial ring over a field $\sk$. Then the ideal $I(G)=\langle X_iX_j:X_i \text{ is adjacent to } X_j\rangle$ is called the {\it edge ideal} of $G$. The Alexander dual of $I(G)$, denoted as $J(G)$, is called the \emph{cover ideal} of $G$. The main problem of interest in this article is to study algebraic invariants such as Hilbert series, regularity and Betti numbers of powers of cover ideals associated to certain graphs. 
The regularity of a monomial ideal is an important invariant
in commutative algebra and algebraic geometry and it measures the complexity of its minimal free resolution. It is known that for a homogeneous ideal $I\subset S$, $\reg(I^s)$ is a linear function of $s$ for $s\gg0$, i.e., there exist non-negative integers $a,b$
and $s_0$ such that $\reg(I^s)=as+b$ for all $s \geq s_0$. This result was proved by Cutkosky, Herzog and Trung \cite{CH} and independently by Kodiyalam \cite{VK}. While the constant $a$ is given by the maximum degree of minimal generators of $I$. On the other hand, $b$ and $s_0$ are not well understood and problem of finding their values 
is addressed by many authors, see \cite{AB,HCJ,TT,AN,ANS,AVT}. The problem of computing the regularity or finding bounds
on the regularity is a difficult problem. Thus one would like to provide bounds and give an explicit formula for the regularity of ideals associated to certain graphs (edge ideals, cover ideals). In the case of edge ideals and cover ideals, the regularity has been studied by various authors, e.g., see \cite{AB,AN,ANS,K}. Thus it is interesting to study the regularity of powers of cover ideals associated to certain graphs. 

Another object of interest is the Rees algebra of an ideal. The \emph{Rees algebra} of a homogeneous ideal $I\subset S$ is a bigraded algebra defined as $\mathcal{R}(I)=\oplus_{s\geq 0}I^st^s$. Rees algebra helps to study the asymptotic behaviour of an ideal and useful in computing the integral closure of powers of an ideal. Rees algebra of an ideal $I$ provides a condition such that $I^s$ has a linear resolution for all $s \geq 1$. 
R\"{o}mer in \cite{Ro} gives an upper bound for the regularity of powers of a homogeneous ideal in terms of $x$-regularity of corresponding Rees algebra $\mathcal{R}(I)$. In particular, if  $x$-regularity of $\mathcal{R}(I)$ is zero, then each power of $I$ has a linear resolution. For a homogeneous ideal $I$, the defining ideal of  $\mathcal{R}(I)$ is studied by many authors, see \cite{H,Hu,MU}, and D. Taylor in \cite{Ta} studied it for a monomial ideal. Further, Villarreal in  \cite{Vi} gives an explicit description of the defining ideal of the Rees algebra of any squarefree monomial ideal generated in degree $2$. Authors in \cite{KS,CS} study the generalized Newton dual of a monomial ideal $I$, and establish an isomorphism between the special fiber rings of $I$ and its generalized dual $\hat{I}$. In this paper, we observe that Rees algebras of $I$ and $\hat{I}$ need not be isomorphic and in a special case we give an explicit description of the defining ideal of the Rees algebra of $\hat{I}$ (see Proposition \ref{ReesGenProp}).

We now give a brief overview of the paper. Section \ref{Prelim} covers some basics of graph theory and commutative algebra which are used throughout the paper.

In Section $3$, we study Rees algebras of cover ideals of certain graphs.  For a squarefree monomial ideal $J$ generated in degree $n-2$, we associate a graph $G_J$. Further, if $G_J$ is a connected graph, then we find the Betti number of $J$ in terms of combinatorial invariants of $G_J$ (see Theorem \ref{BettiThm}). Also, we study the Rees algebra of $J$ using the combinatorial properties of $G_J$. In particular, we prove that if $G_J$ is a tree (unicyclic graph with an odd length cycle), then $\mathcal{R}(J)$ is a quadratic complete intersection (almost complete intersection). Hence  $J^s$ has a linear resolution for all $s \geq 1$, when $G_J$ is a tree.  As a consequence, if $J$ is the cover ideal of the complement graph of a tree, then the Hilbert series and Betti numbers of $J^s$ do not depend on the choice of tree.

In Section $4$, we give an combinatorial formula to compute the Betti numbers of powers of cover ideal of a graph $G$, where $G$ is either a complete graph or a complement graph of a tree. In Section $5$, we compute the regularity of powers of the cover ideals of complete multipartite graphs. Hence, we settle Conjecture 4.10 and 4.11 given by A. V. Jayanthan and N. Kumar in \cite{AN}.

\section{Preliminaries}\label{Prelim}
Let $S=\sk[X_1,\ldots,X_n]$. We use the following notation in this article.
\begin{notation} \hfill{} {\rm
\begin{enumerate}[a)]
\item $[n]=\{1,\dots, n\}$, $n\in \mathbb{N}$.
\item $\Mon(S)=$ set of all monomials in $S$.
\item For a monomial ideal $I$, we denote $M(I)$ be the minimal generating set of monomials of $I$

\item $w(X_i)=$ weight assigned on each variable $X_i$ in $S$. For $u=X_1^{a_1}\cdots X_n^{a_n} \in \Mon(S)$, the 
 degree of $u$ is given by $\deg_w(u)=\sum_{i=1}^{n}a_iw(X_i)$. Further, if $w(X_i)=1$ for all $i$, then we use $\deg(u)$ for $\deg_w(u)$.
\item For $u=X_1^{a_1}\cdots X_n^{a_n} \in \Mon(S)$, we denote $m_i(u)=a_i$.
\item For $v \in \Mon(S)$, we denote $\supp(v)=\{i:X_i \text{ divides } v\}$.
\item For any non empty subset $F \subset [n]$, we set $X_F=\prod\limits_{i\in F}X_i$.
\end{enumerate} }
\end{notation}
\begin{definition}{\rm\hfil{}\\\vspace*{-.4cm}
	\begin{enumerate}[i)]
		\item A finite simple \emph{graph} is an ordered pair $(V,E)$, where $V$ is a collection of \emph{vertices} and $E$ is a collection of subsets of $V$ with cardinality $2$. The elements of $E$ are called the \emph{edges} of a graph $G$. We assume that $V=[n]$.
		\item Let $G=(V,E)$ be a graph on the vertex set $V$. Then the \emph{complement graph} of $G$, denoted by $G^c$, is a graph on vertices $V$ such that $\{i,j\}$ is an edge of $G^c$ if and only if $\{i,j\}\notin E$.
		\item A graph $G=(V,E)$ is called a \emph{complete graph} if for every $i,j\in V$, we have $\{i,j\}\in E$. 
		\item A subset $C\subset [n]$ is called a \emph{cover set} of $G$ if for every $\{i,j\}\in E$, we have $\{i,j\}\cap C\neq \phi$. This set is called a \emph{minimal cover set} if for any $l \in C$, $C\setminus\{l\}$ is not a cover set of $G$.
		\item An \emph{independent set} of a finite graph $G$ is a subset $S \subset V$ such that $\{i,j \} \notin E$ for all $i,j \in S$. The \emph{independent complex} of a finite graph $G$ is a simplicial complex $\Delta(G)$ on $V$ whose facets are the maximal independent subsets of $G$.
		\item A \emph{cycle} of length $\ell$ in a graph $G$ is a subgraph $L$ of $G$ with edge set $$E(L)=\{\{i_1,i_2 \},\{i_2,i_3\},\ldots,\{i_{\ell},i_1\}  \}$$ such that $i_r \neq i_s$ for $r \neq s$.
		\item A \emph{chord} of a cycle $L$ in a graph $G$ is an edge $\{ i,j \}$ of $G$ such that $\{ i,j \}$ is not an edge of $L$ and $\{i,j\}\subset V(L)$. A graph $G$ is called a \emph{chordal graph}, if any cycle in $G$ of length $> 3$ has a chord.
	\end{enumerate}
}\end{definition}
In the following example, we illustrate the minimal cover sets of a graph $G$.
\begin{example}\label{example}{\rm\hfill{}\\\vspace*{-.4cm}
	\begin{enumerate}[a)]
		\item Let $G$ be a graph with $V=\{1,2,3,4\}$ and $E=\{\{1,2\},\{2,3\}\{3,4\},\{4,1\}\}$.  
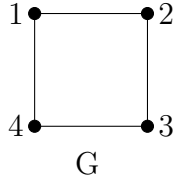
	\[	\begin{tikzpicture}[baseline=(current bounding box.north), level/.style={sibling distance=50mm}]		
		\draw[-] (0,0) rectangle (1.5,1.5);
		\draw [fill] (0,0) circle [radius=0.08];
		\draw [fill] (0,1.5) circle [radius=0.08];		
		\draw [fill] (1.5,1.5) circle [radius=0.08];		
		\draw [fill] (1.5,0) circle [radius=0.08];		
		\node [left] at (0,1.5) { 1};
		\node [right] at (1.5,1.5) { 2};
		\node [left] at (0,0) { 4};	
		\node [right] at (1.5,0) { 3};
		\node[below] at (0.7,-.2){G};	
		\end{tikzpicture}
		\] 
		In this case, minimal cover sets of $G$ are $\{1,3\},\{2,4\}$.	
		\item Let $G$ be the complete graph on the vertex set $[n]$. Then $C$ is a minimal cover set of $G$ if and only if there exists $i \in [n]$ such that $C=[n]\setminus \{i\}$ .
	\end{enumerate}
}\end{example}
\begin{definition}[Edge and Cover Ideals]\label{coverIdeal}{\rm
	Let $G=(V,E)$ be a graph on the vertex set $V=[n]$ and $S=\sk[X_1,\dots,X_n]$. 
	\begin{enumerate}[i)]
		\item The \emph{edge ideal} of $G$, denoted by $I(G)$, is defined as
		$$I(G)=\left\langle X_iX_j: \{i,j\} \in E\right\rangle.$$
		\item The \emph{cover ideal} of $G$, denoted by $J(G)$, is defined as
		$$J(G)=\left\langle \prod\limits_{i\in C}X_i: C \text{ is a minimal cover set  of  } G\right\rangle.$$
	\end{enumerate}
}\end{definition}

\begin{example}{\rm\hfil{}\\\vspace*{-.4cm}
	\begin{enumerate}[a)]
		\item Let $G$ be as in Example \ref{example}(a). Then $J(G)=\langle X_1X_3,X_2X_4\rangle$ is the cover ideal of $G$.
		\item Let $G$ be the complete graph on the vertex set $V=[n]$. Then the cover ideal of $G$ is $$J(G)=\left\langle \prod\limits_{ j \neq i}X_j: i \in [n]\right\rangle.$$ 
	\end{enumerate}
}\end{example}
\begin{definition}\label{puredef}
	{\rm Let $M$ be a finitely generated graded $S$-module.
		\begin{enumerate}[i)] 
			\item Then $\beta^S_{i,j}(M) = (\dim_\sk(\Tor_i^S(M,\sk))_j$ is called the $(i,j)^{th}$ \emph{graded Betti number} of $M$. 
			\item The \emph{regularity} of $M$, denoted as $\reg(M)$, is defined as $$\reg(M)=\max\{j-i:\beta^S_{i,j}(M)\neq 0\}.$$
			\item The \emph{projective dimension} of $M$, denoted as $\pdim_S(M)$, is defined as $$\pdim_S(M)=\max\{i:\beta^S_{i,j}(M)\neq 0 ~\text {for ~ some}~ j\}.$$
			
			\item Let $\pdim_S(M) = p$. If for each $i\in \{0,1,\ldots,p  \}$, there exists a number $d_i$ such that 
			$\beta_{i,j}^S(M)=0$ for $j \neq d_i$,
			 then $M$ is said to have a {\it pure resolution of type} $(d_0,d_1,\ldots, d_p)$. Further, if $d_0=d$ and $d_i=d+i$ for all $i \in [p]$, then we say that $M$ has a \emph{linear resolution}.
\end{enumerate}
}\end{definition}
\begin{definition}{\rm  \label{quotient}
	Let $I$ be a monomial ideal having linear quotients with respect to some order  $ u_1,\dots, u_r$ of elements of $M(I)$. Then, for $2\leq j\leq r$, we define $I_j=\langle u_1,\dots, u_{j-1}\rangle$  and $\set(u_j)=\{X_k: X_k\in I_j:u_j\}.$ We set $A_t(I^s)=\{u \in M(I^s): |\set(u)|=t \}.$
}\end{definition}
We use the following result of Herzog and Takayama \cite[Lemma 1.5]{HT},  to compute the Betti numbers of certain monomial ideals having linear quotients.
\begin{lemma}[Herzog-Takayama, \cite{HT}]
\label{resRem}
Suppose that $I$ has linear quotients with respect to order $u_1,\ldots, u_r$
of generators of $I$ and $\deg(u_1)\leq \deg(u_2) \leq \dots \leq \deg(u_r)$. Then the iterated
mapping cone $F$, derived from the sequence $u_1,\ldots,u_m$, is a minimal graded free
resolution of $S/I$ and for all $i>0$, the symbols
	$f(\sigma;u)~\text{with}~u\in M(I), \sigma \subset \set(u), |\sigma|= i-1$,
	 forms a basis for $F_i$.
	Moreover,  $\deg(f(\sigma, u))=\deg(u)+|\sigma|$. 
	\end{lemma}

\section{Rees Algebra of a Monomial Ideal}

Let $J\subset S$ be a squarefree monomial ideal generated in degree $n-2$. Then for every $u\in M(J)$, there exist $i,j$ such that $[n]\setminus\supp(u)=\{i,j\}$. Now we associate a graph $G_J$ to the ideal $J$ on the vertex set $[n]$ with edge set  $\{\{i,j\}: [n]\setminus\supp(u)=\{i,j\} \text{ for some } u \in M(J)\}.$ In this section, we discuss the Rees algebra $\mathcal{R}(J)$ and give an explicit description of the defining ideal of $\mathcal{R}(J)$ in terms of properties of $G_J$. For better understanding of defining ideal of $\mathcal{R}(J)$, we compute the Betti numbers of $J$. In particular, we find the defining ideal of the Rees algebra of the cover ideal of a graph whose complement graph is triangle free.

\begin{definition}{\rm
		Let $I$ be a homogeneous ideal of $S$. Then the \emph{Rees algebra} of $I$ is defined as $S[It]=\bigoplus\limits_{s\geq 0}I^st^s$, and it is denoted by $\mathcal{R}(I)$. 
		
		Let $I=\langle u_1,\dots, u_r\rangle$ and $R=S[T_1,\dots, T_r]$, then there is a surjective ring homomorphism $\phi:R\longrightarrow \mathcal{R}(I)$ by setting $\phi(X_i)=X_i$ for $i=1,\ldots,n$ and $\phi(T_j)=u_jt$ for $j=1,\ldots,r$. Set $K=\ker(\phi)$. The ideal $K$ is called the \emph{defining ideal} of the Rees algebra. Further, note that $K=\bigoplus\limits_{i\geq 1} K_i$, where $K_i$ is a homogeneous component of degree $i$ in $T_j$-variables. If $K=RK_1$, then $I$ is said to be of \emph{linear type}. 
}\end{definition}

\begin{notation}{\rm
		Let $I$ be a monomial ideal generated by $u_1,\dots, u_r $ and $I_s$ be a set of all non-decreasing sequences $\alpha=(i_1,\ldots,i_s)$ in $[r]$ of length $s$. Then for any $\alpha\in I_s$ we denote $u_\alpha =u_{i_1}u_{i_2}\cdots u_{i_s}$, $T_\alpha=T_{i_1}T_{i_2}\cdots T_{i_s}$ and $\widehat{u}_{\alpha_l} =u_{i_1}\cdots \widehat{u_{i_\ell}}\cdots u_{i_s}$. For any $\alpha,\beta\in I_s$, we define $$T_{\alpha,\beta}=\dfrac{\lcm[u_\alpha, u_\beta]}{u_\beta}T_\beta-\dfrac{\lcm[u_\alpha, u_\beta]}{u_\alpha}T_\alpha$$.
}\end{notation}
Defining ideal of the Rees algebra of a monomial ideal is studied by D. Taylor in \cite{Ta}. In order to prove the Proposition \ref{ReesGenProp}, we use the following result.
\begin{theorem}[Taylor, \cite{Ta}]\label{ReesGenThm}
	Let $S=\sk[X_1,\dots, X_n]$ and $I=\langle u_1,\dots, u_r\rangle $ be a monomial ideal in $S$. Then $\mathcal{R}(I)\simeq R/K$, where $R=S[T_1,\ldots,T_r], K=RK_1+R(\bigcup\limits_{s=2}^{\infty} K_s)$ with $K_s=\{T_{\alpha,\beta}: \alpha,  \beta\in I_s\}$.
\end{theorem}
\begin{notation}{\rm
		Let $S=\sk[X_1,\dots, X_n]$  and $u$ be a squarefree monomial in $S$. Set $\X=X_1\cdots X_n$. We denote 
		the monomial $\dfrac{\X}{u}$  by $u'$.
}\end{notation}
\begin{remark}\label{lcmRem}{\rm
		Let $u,v$ be monomials in $S$. Then we have $\dfrac{\lcm[u',v']}{v'}=\dfrac{\lcm[u,v]}{u}$.
}\end{remark}

In the following proposition, we extend the result \cite[Theorem 3.1]{Vi} of Villarreal for any squarefree monomial ideal generated in degree $n-2$.
\begin{proposition}\label{ReesGenProp}
	Let $J=\langle u_1,\dots, u_r \rangle$ be a squarefree monomial ideal of  $S$ generated in degree $n-2$ and $K$ the defining ideal of the Rees algebra of $J$. Then $K=RK_1+R(\bigcup\limits_{s=2}^\infty P_s)$, where $P_s=\{T_\alpha-T_\beta:u_\alpha=u_\beta \text{ for some } \alpha,\beta\in J_s\}.$
\end{proposition}

\begin{proof} For $s\geq 2$,
	let $\alpha=(i_1,\dots, i_s)$ and $\beta=(j_1,\dots, j_s)\in J_s$ such that $\alpha\neq\beta$. From the proof of \cite[Theorem 3.1]{Vi}, it follows that there exist integers $l$ and $m$ and a monomial $v$ such that $\lcm[u'_\alpha, u'_\beta]=u'_{i_l}\widehat{u}'_{\beta_m}v$. This implies that $\dfrac{\lcm[u'_\alpha,u'_\beta ]}{\lcm[u'_{i_l},u'_{j_m}]\widehat{u'}_{\beta_m}}$ and $\dfrac{\lcm[u'_\alpha,u'_\beta]}{\lcm[\widehat{u'}_{\alpha_l},\widehat{u'}_{\beta_m}]u'_{i_l}}$ are monomials in $R$.
	By Theorem \ref{ReesGenThm}, we know that $K_s$ is generated by  polynomials of type 
	$$T_{\alpha,\beta}=\dfrac{\lcm[u_\alpha, u_\beta]}{u_\beta}T_\beta-\dfrac{\lcm[u_\alpha, u_\beta]}{u_\alpha}T_\alpha.$$
	By Remark \ref{lcmRem}, we know that $T_{\alpha,\beta}=\dfrac{\lcm[u'_\alpha, u'_\beta]}{u'_\alpha}T_\beta-\dfrac{\lcm[u'_\alpha, u'_\beta]}{u'_\beta}T_\alpha.$ Now, one can note that 
	$T_{\alpha,\beta}=AT_{j_m}+B\widehat{T}_{\alpha_l}$, where
	
	$$A=\dfrac{\lcm[u'_\alpha,u'_\beta ]} {\lcm[\widehat{u'}_{\alpha_l},\widehat{u'}_{\beta_m}]u'_{i_l}} \left(\dfrac{\lcm[\widehat{u'}_{\alpha_l},\widehat{u'}_{\beta_m}]}{\widehat{u'}_{\alpha_l}}\widehat{T}_{\beta_m}-\dfrac{\lcm[\widehat{u'}_{\alpha_l},\widehat{u'}_{\beta_m}]}{\widehat{u'}_{\beta_m}}\widehat{T}_{\alpha_l}\right)$$ and 
	$$B=\dfrac{\lcm[u'_\alpha,u'_\beta]} 
	{\lcm[u'_{i_l},u'_{j_m}]\widehat{u'}_{\beta_m}}\left(\dfrac{\lcm[u'_{i_l},u'_{j_m}]}{u'_{i_l}}T_{j_m}-\dfrac{\lcm[u'_{i_l},u'_{j_m}]}{u'_{j_m}}T_{i_l}\right).$$

	Note that $A\in K_{s-1}$ and $B\in K_1$, and hence we get $K_s\subset R_1K_{s-1}+R_{s-1}K_1$. The backward mathematical induction completes the proof.
\end{proof}

\begin{corollary}\label{linTypCor}
	Let $J$ be a squarefree monomial ideal generated in degree $n-2$ and $G_J$ be the associated graph. Then $J$ is of linear type if and only if $G_J$ is a forest or it has a unicycle with an odd length cycle. 
\end{corollary}
\begin{proof}
	From Proposition \ref{ReesGenProp}, it can be seen that $J$ is of linear type if and only if for any $s\geq 2$, we have $u_\alpha\neq u_\beta$ for all $\alpha, \beta\in J_s$ with $\alpha \neq \beta$. Note that $u_\alpha=u_\beta$ if and only if $u'_\alpha=u'_\beta$. It follows from \cite[Corollary 3.2]{Vi} $u'_{\alpha}\neq u'_{\beta}$ for all $\alpha,\beta\in J_s$ if and only if $G_J$ is a forest or it is a unicycle with an odd length cycle.
\end{proof}

In order to understand the number of generators of $K$, we need to compute the Betti numbers of $J$ and closed even walks of $G_J$. In the following, we compute the Betti numbers of $J$.

\begin{theorem}\label{BettiThm}
	Let $J\subset S$ be a squarefree monomial ideal generated in degree $n-2$ such that the associated graph $G_J$ is connected. Then $J$ has a linear quotient. Moreover, if $G_J$ has $r$ number of edges and $n$ number vertices, then
	$$\beta_{i,j}^S(J)=\begin{cases}
	r &\text{if}~ (i,j)=(0,2),\\
	2r-n&\text{if}~ (i,j)=(1,3),\\
	r-n+1& \text{if}~ (i,j)=(2,4),\\
	0 & \text{otherwise.}
	\end{cases}$$
\end{theorem}
\begin{proof}
	Think $G_J$ as a $1$-dimensional simplicial complex, say $\Delta(G_J)$. Since connected $1$-dimensional simplicial complex is always shellable, so is $\Delta(G_J)$. Therefore, $I(\Delta(G_J)^\vee)$ has a linear quotient, where $\Delta(G_J)^\vee$ is the Alexander dual of $\Delta(G_J)$. Thus, result follows from the fact that $I(\Delta(G_J)^\vee)=J$. 
	
	Let $H$ be a spanning tree of $G_J$. Set $e_1$ be the edge incident to a leaf vertex. Since $H$ is connected graph, there exists an edge, say $e$, with $e\cap e_1\neq \phi$. Set $e=e_2$. In the similar manner, for  $2\leq i\leq n-1$, set $e_i$ to be an edge of $H$ such that $e_i\cap e_j\neq \phi$ for some $j<i$. Since the graph $H$ has no isolated vertices,  $\cup_{i=1}^{n-1}e_i\cap e=e$, where $e$ belongs to the $E(G_J) \setminus E(H)$ . Therefore, $e_1,\dots, e_r$, where $e_1,\dots ,e_{n-1}$ are edges of $H$, gives a shelling on $\Delta(G_J)$.  Now, ordering  $\{X'_{e_1},\dots, X'_{e_r}\}$ gives the linear order on generators of $J$. 
	
	In order to find $\set(X'_{e_j})$, we need to understand $X'_{e_j}:X'_{e_i}$. It is easy to see that $X'_{e_j}:X'_{e_i}=X'_{e_j}/\gcd(X'_{e_j},X'_{e_i})=\lcm(X_{e_j},X_{e_i})/X_{e_j}.$ Thus, $X'_{e_j}:X'_{e_i}$ is either $X_{e_i}$ or a variable whose index is a vertex in $e_j$ but not in $e_i$. For $2\leq i\leq n-1$, let $H'$ be the subgraph of $H$ with edge set $\{e_1,\dots, e_i\}$. Then note that one of a vertex of $e_i$ in $H'$ is a leaf vertex, say $i_1$. This implies that $\cup_{k=1}^{i-1}e_k\cap e_i=\{i_1\}$. Thus, for $2\leq i\leq n-1$, we have $|\set(X_{e_i}')|=1$. Now, for $n\leq i\leq r$, both vertices of the edge $e_i$ belong to different edges of $H$, and hence $|\set(X_{e_i}')|=2$. Therefore, using Lemma \ref{resRem}, we get $\beta_{0,2}^S(J)=r$, $\beta_{1,3}^S(J)=n-2+2(r-n+1)=2r-n$ and $\beta_{2,4}^S(J)=r-n+1$. 
\end{proof}

From the above theorem and Euler's formula for planar graph, we get the following.

\begin{corollary}
	Let $J\subset S$ be a squarefree monomial ideal generated in degree $n-2$ such that the associated graph $G_J$ is a connected planar graph. If $G_J$ has  $m$ number of bounded regions, then
	$$\beta_{i,j}^S(J)=\begin{cases}

	n+m-1 &\text{if}~ (i,j)=(0,2),\\
	n+2m-2 &\text{if}~ (i,j)=(1,3),\\
	m& \text{if}~ (i,j)=(2,4),\\
	0 & \text{otherwise.}
	\end{cases}$$
\end{corollary}

\begin{lemma}\label{comIntReesLem}
	Let  $J$ be a squarefree monomial ideal generated in degree $n-2$. If $G_J$ is a tree (unicyclic graph with an odd length cycle), then $\mathcal{R}(J)$ is a quadratic complete intersection (almost complete intersection). In particular, when $G_J$ is a tree, $J^s$ has a linear resolution for $s\geq 1$.
\end{lemma}
\begin{proof}
	Let $R=S[T_1,\dots, T_{n-1}]$ and $\mathcal{R}(J)\simeq R/K$. By Corollary \ref{linTypCor}, we know that $J$ is of linear type, and hence $K$ is generated by $\beta_1^S(J)=n-2$ elements. Note that $\dim(R)=2n-1$ and $\dim(\mathcal{R}(J))=n+1$. This implies that $\hgt(K)=n-2=\mu(K)$ which forces that $K$ is generated by a regular sequence, and hence $\mathcal{R}(J)$ is a complete intersection.
	
	Further, $J$ is linearly presented and of linear type implies that $K$ is generated in degree $(1,1)$. Since $K$ is generated by a regular sequence of degree $(1,1)$, we get that $\reg_x(\mathcal{R}(J))=0$. Now, \cite[Proposition 10.1.16]{HH} completes the proof.
\end{proof}

\begin{lemma}\label{hilbLem}
	Let $T_1$ and $T_2$ be trees on $n$ vertices and $J_1$ and $J_2$ be cover ideals of their complement graphs, respectively. 
	Then $H_{J_1^s}(z)=H_{J_2^s}(z)$ for all $s\geq 1$. 
\end{lemma}
\begin{proof}
	Note that if $H_{\mathcal{R}(I)}(z_1,z_2)$ is the bigraded Hilbert series of the Rees algebra of a monomial ideal $I$, then $$H_{I^s}(z_1)=\left.\dfrac{1}{s!}\dfrac{\partial^s}{\partial z_2^s} \left(  H_{\mathcal{R}(I)} (z_1,z_2)\right)\right\vert_{z_2=0}.$$ Thus it is enough to prove that $H_{\mathcal{R}(J_1)}(z_1,z_2)=H_{\mathcal{R}(J_2)}(z_1,z_2)$. This follows from the proof of Lemma \ref{comIntReesLem}.
\end{proof}

\begin{corollary}\label{sameBettiCor}
	Let $J_1$ and $J_2$ be as in Lemma \emph {\ref{hilbLem}}. Then $\beta^S(J_1^s)=\beta^S(J_2^s)$ for all $s\geq 1$.
\end{corollary}
\begin{proof}
	The proof of the corollary follows from Theorem \ref{BettiThm} and Lemma \ref{hilbLem}.
\end{proof}

The above corollary says that if $G$ is the complement graph of a tree, then the Betti numbers of powers of cover ideal do not depend  on tree. However, if $G$ is the complement of a unicyclic graph, then above result does not hold. For example, let $G$ be the complement graph of a cycle $C_7$ and $H$ be the complement of $\{\{1,2\},\{2,3\},\{3,4\},\{4,5\},\{5,1\},\{5,6\},\{6,7\}\}$. Then using Macaulay2  \cite{M2}, we see that $$196=\beta^S_{2,16}\left(\frac{S}{J(G)^3}\right)\ne \beta^S_{2,16}\left(\frac{S}{J(H)^3}\right)=195.$$

\section{Betti Numbers of Powers of Cover Ideals}
\noindent Let $K_n$ be a complete graph on $n$ vertices and $J(K_n)$ be its cover ideal. In this section our goal is to compute the Betti numbers of $J(K_n)^s$ for $s \geq 1$ which proves \cite[Conjecture 4.10]{AN}. As a consequence, we find the Betti numbers of powers of cover ideal of the complement of a tree. Now for $s\geq 1$, observe that 
\begin{align*}
M(J(K_n)^s) =\left\{ \frac{\mathbf{X}^s}{w}: \mathbf{X}=\prod_{i=1}^n X_i, w \in \Mon(S), \deg(w)=s \right\}. 
\end{align*}

The following result is a special case of Theorem 4.2. in \cite{KS}. 
\begin{lemma}\label{linQuoLem}  For $s\geq 1$, the ideal $J(K_n)^s$ has linear quotients with respect to the reverse lexicographic order of the generators.
\end{lemma}
	 
Let  $M(J(K_n)^s)=\{u_1,\ldots,u_r\}$, where $u_1>u_2>\cdots>u_r$ in the reverse 
lexicographical order with respect to $X_1 > X_2 > \cdots> X_r$. Then for $2\leq j \leq r$, we compute the $\set(u_j)$, which will be useful in proving Theorem \ref{comBettiThm}. 

\begin{lemma}\label{LemSet}
	 For $2\leq j \leq r$, let $u_j=\dfrac{\mathbf{X}^s}{v}$, where $v\in \Mon(S)$, $\deg(v)=s$. Then $\set(u_j)=\{X_i:i<n ~and~ i \in \supp(v)\}$. 
\end{lemma}
\begin{proof}
Let $X_i \in \set(u_j)$, where $1 \leq i \leq n$. Then $X_iu_j\in \langle u_1,\dots, u_{j-1}\rangle$. In particular, there exists $v_1 \in \Mon(S)$ such that $X_iu_j=u_lv_1$, for some $1 \leq l \leq {j-1}$. Note that $\deg(u_l)=\deg(u_j)$ implies that $\deg(v_1)=1$, hence $v_1=X_{j_i}$ for some $j_i\in [n]$. Thus, we get $\dfrac{X_iu_j}{X_{j_i}}=u_l \in M(J(K_n)^s).$ This implies that that $m_i\left(\dfrac{X_iu_j}{X_{j_i}}\right)\leq s$, and hence $m_i\left(\dfrac{u_j}{X_{j_i}}\right)\leq s-1$. Since $X_i\neq X_{j_i}$, we get $i\in \supp(v)$. Suppose $i=n$. Then $X_nu_j=u_lX_{j_i}$ implies that $m_n(u_l)>m_n(u_j)$, which is a contradiction to the fact that $u_l >_{revlex}u_j$.

Conversely, let $i \in \supp(v)$, $i\neq n$. Since $j\neq 1$, we know that $X_n | u_j$. Now, consider a monomial $u'=\frac{X_iu_j}{X_n}$. Clearly, $u'>_{revlex}u_j$ and $X_iu_j=u'X_n$. This completes the proof.
\end{proof}
Now we are in position to compute the Betti numbers of $J(K_n)^s$. 
 
\begin{theorem}\label{comBettiThm} The Betti numbers of $J(K_n)^s$ are given by $$ \beta_i^S(J(K_n)^s)=\binom{n-1}{i}\binom{n-1-i+s}{n-1}. $$
\end{theorem}
\begin{proof}
Firstly, we show that $|A_t(J(K_n)^s)|= \binom{n-1}{t}\binom{s}{t}$. In view of Lemma \ref{LemSet}, we get 
$$ A_t(J(K_n)^s)=\left\{\dfrac{\mathbf{X}^s}{v}:v \in \Mon(S), \deg(v)=s ~\text{and}~ |\supp(v)\cap [n-1] |=t\right\}.$$ Since each monomial $v \in A_t(J(K_n)^s)$ corresponds to a unique monomial $v'$ in $(n-1)$ variables of degree less than equal to $s$ with $|\supp(v')|=t$, we get $|A_t(J(K_n)^s)|= \binom{n-1}{t}\binom{s}{t}$.
Now using Lemma \ref{resRem}, we get 
\begin{equation*}
\beta_i^S(J(K_n)^s)=
 \sum_{t=0}^{n-1}\binom{n-1}{t}\binom{s}{t}\binom{t}{i} 
 = \binom{n-1}{i}\binom{n-1-i+s}{n-1},~~ 
\end{equation*}
where the last equality follows from  the Chu-Vandermonde identity \cite[Page 26]{Co}. 
\end{proof}
As an immediate consequence, we get the following result.
\begin{theorem}\label{treBettiThm}
	Let $T$ be a tree on $n$ vertices and $G$ the complement graph of $T$. Let $J(G)$ be the cover ideal of $G$. Then the Betti numbers of $J(G)^s$ are given by
	$$\beta_i^S(J(G)^s)=\binom{n-2}{i}\binom{n-2-i+s}{n-2}. $$
\end{theorem}
\begin{proof}
	By Corollary \ref{sameBettiCor}, we may assume that $T$ be a star graph, and hence its complement is a complete graph on $n-1$ vertices. Thus the result follows from Theorem \ref{comBettiThm}.
\end{proof}
\section{Regularity of Powers of Cover Ideals of Complete Multipartite Graphs}
\noindent
In this section, our goal is to prove \cite[Conjecture 4.11]{AN}.
Let $J(G)$ be the cover ideal of a complete $n$-partite graph $G$ on the vertex set $V$ with partition $V_1\sqcup \cdots \sqcup V_n$, where $V_i=\{X_{i,j} \}$, $1 \leq j \leq w_i.$ Then by taking $X_i=\prod_{j=1}^{w_i}X_{i,j}$, one can identify $J(G)$ with the cover ideal of a complete graph on vertices $X_1,\dots,X_n$. We set $\deg(X_{i,j})=1$ and $w(X_i)=w_i$. Thus to compute the regularity of powers of the cover ideal  of a complete multipartite graphs, we compute the regularity of powers of the cover ideal of a complete graph $K_n$ on vertices $X_1,\dots,X_n$ with $w(X_i)=w_i$.

\begin{notation}{\rm 
	Let $\sigma\subset \{X_1,\dots, X_n\} $ and $u$ be a monomial in $S$. Then we denote $$m(u,\sigma)= u\prod_{X_k\in \sigma}X_k.$$
	}
\end{notation}
\begin{remark}\label{remSetDeg}{\rm Let $J(K_n)$ be the cover ideal of a complete graph $K_n$ on vertices $X_1,\ldots,X_n$ and $M(J(K_n)^s)=\{u_1,\ldots,u_r\}$, where $u_1>u_2>\cdots>u_r$ in the reverse 
lexicographical order. Further, if we assume  $w(X_i)=w_i$ with $w_i\leq w_j$ for all $i\leq j$, then observe the following. 
		\begin{enumerate}[i)]
			\item If $u\in M(J(K_n)^s)$ and $\sigma \subset \set(u)$, then by Lemma \ref{LemSet}, we get that for any $i$, $X_i^{s+1}$ does not divide $m(u,\sigma)$. In other words, we have $m_i(m(u,\sigma))\leq s,$ for all $i$.
			 Further, if $|\sigma|=i-1$, then $\deg_w(m(u,\sigma))\leq w_1(i-1)+s\sum_{l=2}^{n}w_l$.
			 \item The proof of Lemma \ref{resRem} remains valid even if we assign weight $w(X_i)=w_i$ for each $i$. Also in this case, $\deg_w(f(\sigma, u))=\deg_w(u)+\sum\limits_{X_i \in \sigma}w_i$.
			 \item  It follows from Lemma \ref{resRem} and Lemma \ref{LemSet} that 
			 $\pdim_S(S/J(K_n)^s)=\begin{cases}
s+1 & \text{if}~ s<n-1,\\
n & \text{otherwise}.
\end{cases}$
		\end{enumerate}
	
}\end{remark}

Now we proceed to calculate the regularity of $S/J(K_n)^s$ in the above set-up.

\begin{theorem}
 Let $S=\sk[X_1,\dots,X_n]$ with $w(X_i)=w_i$ and  $K_n$ be a complete graph on vertices $X_1,\dots,X_n$. Further, if we assume $w_1 \leq \dots \leq w_n$, then 
 $$ \reg(S/J(K_n)^s)=\begin{cases}
s\sum_{l=1}^{n}w_l-(s+1)  & \text{if}~ s<n-1,\\
s\sum_{l=1}^{n}w_l-w_1(s-n+1)-n  & \text{otherwise}.
\end{cases}$$ 
\end{theorem} 

\begin{proof}
 Firstly, using Lemma \ref{linQuoLem} we get $J(K_n)^s$ has  linear quotients with reverse lexicographic order. Now, Lemma \ref{resRem} implies that a basis element of $i$th component $F_i$ of a graded minimal free resolution $\F$ of $S/J(K_n)^s$ is given as following:
$$f(\sigma;u)~\text{with}~u\in M\left(J(K_n)^s\right), \sigma \subset \set(u), |\sigma|= i-1,$$
where $\deg_w(f(\sigma; u))=\deg_w(m(\sigma; u)).$ This implies that $\beta^S_{i,j}\left(S/J(K_n)^s \right)\neq 0$ if and only if there exists some $\sigma \subset \set(u)$ with $|\sigma|=i-1$ such that $j=\deg_w(m(\sigma; u))$.  Let $d_i=\max\{ {j:\beta_{i,j}\left( S/J(K_n)^s  \right)\neq 0} \}$. Then
$$d_i=\max \{\deg_w(m(\sigma,u)):\sigma \subset \set(u), |\sigma|= i-1 ~\text{and}~ u\in M\left(J(K_n)^s\right) \}.$$ By Remark \ref{remSetDeg}(ii), it is easy to see that $s \geq i-1$. Now, for $i<n$ take $u=\dfrac{X^s}{X_1^{s-(i-1)}X_2\cdots X_i}$ with $\sigma=\{X_2,\ldots,X_i  \}$, and for $i=n$
 take $u=\dfrac{X^s}{X_1^{s-(i-2)}X_2\cdots X_{i-1}}$ with $\sigma=\{X_1,\ldots,X_{i-1}  \}$. Now note that in the both cases $m(\sigma,u)=X_1^{i-1}X_2^s \cdots X_n^s$, and hence $\deg_w(m(\sigma,u))=w_1(i-1)+s\sum_{l=2}^{n}w_l$. Thus Remark \ref{remSetDeg}(i) gives $d_i=w_1(i-1)+s\sum_{l=2}^{n}w_l$. Note that $d_i-i \leq d_{i+1}-(i+1)$. Hence from Remark \ref{remSetDeg}(ii), the result follows.
\end{proof}
As an immediate consequence, we get the following corollary.
\begin{corollary}
	Let $G$ be a complete multipartite graph with partition on the vertex set $V_1\sqcup\cdots \sqcup V_n$. If $|V_i|=w_i$ with $w_i\leq w_{i+1}$, then 
	$$ \reg(S/J(G)^s)=\begin{cases}
	s\sum_{l=1}^{n}w_l-(s+1)  & \text{if}~ s<n-1,\\
	s\sum_{l=1}^{n}w_l-w_1(s-n+1)-n  & \text{otherwise}.
	\end{cases}$$
\end{corollary}

\renewcommand{\bibname}{References}

\end{document}